\documentclass[12pt]{article}
\usepackage{amsmath, amsthm, amssymb,amsfonts}
 \usepackage[utf8]{inputenc}
\usepackage{mathtools}
\makeatother
\usepackage{enumerate}
\usepackage{enumitem}  
\usepackage{hyperref}
\usepackage[margin=1in]{geometry}
\newtheorem{theorem}{Theorem}[section]
\newtheorem{thm}{Theorem}

\newtheorem{lemma}[theorem]{Lemma}
\theoremstyle{definition}
\newtheorem{remark}[theorem]{Remark}
\newcommand{\Stab}{\mathrm{Stab}}

\newcommand{\PGL}{\mathrm{PGL}}
\newcommand{\Sym}{\mathrm{Sym}}
\newcommand{\Alt}{\mathrm{Alt}}
\newcommand{\M}{\mathrm{M}}
\newcommand{\GF}[1]{\mathbb{GF}(#1)}
\newcommand{\PP}[1]{\mathbb{P}(#1)}
\renewcommand{\P}[1]{\mathcal{P}(#1)}
\renewcommand{\L}[1]{\mathcal{L}(#1)}
\renewcommand{\t}[1]{\tau_{#1}}
\newcommand{\h}[1]{h_{({#1})}}
\newcommand{\x}[1]{x_{({#1})}}
\newcommand{\Pseudo}[1]{\frak{G}(#1)}
\newcommand{\PseudoG}[1]{\mathcal{G}(#1)}
\begin{document}
\title{\Huge{ $\PP{q}$-groupoids of Conway type}}
\author{Veronica Kelsey\footnote{Email address of corresponding author: veronica.kelsey@manchester.ac.uk  \newline Key words: groupoids, projective planes \newline MSC: 20B25}, Peter Rowley \\
 \small{Department of Mathematics, University of Manchester, Oxford Road, M13 6PL, UK}\\ \small{Heilbronn Institute for
Mathematical Research, Bristol, UK}}
 \maketitle
\begin{abstract}
In the spirit of Conway \cite{Conwaysurvey} we define a groupoid starting from projective planes of order $q$, where $q$ is odd. The associated group of these groupoids is then investigated.

\end{abstract}
\date{}

\section{Introduction}\label{Sec:Intro}
Noyes Palmer Chapman's classic 15-Puzzle, dating back to the 1870s,  consists of 15 numbered tiles and one hole in a $4 \times 4$ tray. In this puzzle, often incorrectly credited to Sam Loyd (see Slocum and Sonneveld \cite{SS}), tiles are slid horizontally or vertically and the goal is to arrange the tiles in numerical order.

Taking inspiration from the 15-Puzzle Conway \cite{Conwaysurvey} proposed a variation based on replacing the $4 \times 4$ tray with the 13 points of the projective plane of order 3. Following on from Conway's idea there have been a number of other variations on this theme involving certain graphs or block designs. For an example of the former using regular $3$-valency graphs see Wilson \cite{Wilson} and Archer \cite{Archer} and for the latter, employing $2-(n,4,\lambda)$ designs, consult Gill, Gillesie, Praeger, Semeraro \cite{GGPS} and Gill, Gillesie, Semeraro \cite{GGS}.  Yet other perspectives may be seen in Ekenta,  Jang, Siehler \cite{EJS}. Here we study a further extension of these ideas but continuing with the theme of projective planes.\medskip 

We also emulate the 15-Puzzle construction. In place of a $4 \times 4$ tray, we consider the points $\P{q}$ of the projective plane $\PP{q}$. For each $\alpha,\beta \in \P{q}$ we will define a unique move $\h{\alpha,\beta}$ which interchanges $\alpha$ and $\beta$, acts as an element of order 2 on the line $\langle \alpha, \beta \rangle$, and fixes the remaining points of $\P{q}$.
Let $\Pseudo{q}$ be the set of elements given by sequences of elementary moves of the form 
\[\h{\alpha_0,\alpha_1,\ldots, \alpha_s}:=\h{\alpha_0,\alpha_1}\h{\alpha_1,\alpha_2}\cdots \h{\alpha_{s-1},\alpha_s}.\]
Two elements $\h{\alpha_0,\ldots, \alpha_s}, \h{\beta_0,\ldots, \beta_t} \in \Pseudo{q}$ are multiplied by concatenation provided $\alpha_s=\beta_0$ to give $\h{\alpha_0,\ldots, \alpha_s, \beta_1,\ldots, \beta_t}$. Since not all possible products are defined, $\Pseudo{q}$ is a groupoid and not a group.

If we fix $\alpha \in \P{q}$, then we have a group $\PseudoG{q}$ defined by
\[\PseudoG{q}: = \{ \h{\alpha,\alpha_1 , \ldots, \alpha_{s-1}, \alpha} \in \Pseudo{q} \} \leq \Sym\left(\P{q} \backslash \{\alpha\}\right) .\] 
 
Taking $q = 3$ recovers Conway's construction \cite{Conwaysurvey} yielding, as he termed it, the pseudogroup $\M_{13}$, which in our notation is $\Pseudo{3}$. Conway shows that for any $\alpha \in \P{3}$, the group $\PseudoG{3}$ is isomorphic to $\M_{12}$, the Mathieu group of degree $12$. Further, in \cite{Conwaysurvey}, Conway also considers what he terms the doubling of $\M_{13}$ which ends up exposing the double cover of $\M_{12}$. In Conway, Elkies and Martin \cite{CEM} (see also Martin \cite{Martin}) there is an extended account of both $\M_{13}$ and its doubling.\medskip

In this paper we determine $\PseudoG{q}$.
\begin{thm}\label{main} Let $q>3$ be an odd prime power, $\alpha$ a fixed point of $\P{q}$ and set $\Omega = \P{q} \backslash \{\alpha\}$. Then
\[\PseudoG{q}= \begin{cases}
\Alt(\Omega) & \text{ if } q \equiv 3 \bmod 4,\\
\Sym(\Omega)& \text{ if }q \equiv 1 \bmod 4.
\end{cases}\]
\end{thm}

The sporadic group $\M_{12}$ has a history of having tendrils in many, and varied, combinatorial areas. See \cite{atlas} and \cite{CS} for a summary of some of these. In relation to $\PseudoG{q}$, yet again $\M_{12}$ stands as a kind of singularity.

\section{Preliminaries}
We begin with some general facts on projective plans.
Let $q$ be a power of an odd prime and $\PP{q}$ denote the projective plane defined over $\GF{q}$. Hence $\PP{q}$ consists of a point set $\P{q}$ and a line set $\L{q}$, each of size $q^2+q+1$.
It is well-known that every distinct pair of points $\beta,\gamma \in \P{q}$ lie on a unique line which we denote by $\langle \beta, \gamma \rangle$. Dually, any two distinct lines $\ell, \ell' \in \L{q}$ intersect in a unique point which we denote by $\ell \cap \ell'$. For a line $\ell\in \L{q}$, let $\Gamma_0(\ell)$ denote the set of points incident with $\ell$. We use $\{e_1,e_2,e_3\}$ to denote the standard basis vectors of $\GF{q}^3$.

\begin{lemma}\label{cyclicaction} 
Let $L =\PGL(3,q)$ act on $\PP{q}$, let $\beta, \gamma \in \P{q}$ be distinct, and let $\ell=\langle \beta,\gamma\rangle$. 
\begin{enumerate}[label=\rm{(\roman*)}]
\item \label{beta,ell} $L_{\beta,\ell}: = \Stab_L(\beta) \cap \Stab_L(\ell)$ has shape $q^3(q-1)^2$ and acts $2$-transitively on $\Gamma_0(\ell) \setminus \{ \beta \}$.
\item \label{beta,gamma}  $L_{\beta,\gamma}: = \Stab_L(\beta) \cap \Stab_L(\gamma)$ has shape $q^2:(q - 1)^2$ with the $q^2(q-1)$ fixing all of $\Gamma_0(\ell)$, and the induced action on $\Gamma_0(\ell) \setminus \{ \beta , \gamma\}$ is cyclic of order $q - 1$.
\end{enumerate}
\end{lemma}
\begin{proof} Since $L$ is 2-transitive on $\P{q}$ we may assume that $\beta=\langle e_1 \rangle$ and $\gamma=\langle e_2 \rangle$. 

The induced elements of $L_{\beta,\ell}$, and of $L_{\beta,\ell}$'s action on $\Gamma_0(\ell)\backslash \{\beta\}$ can be represented by right multiplication of the following matrices where $a,b ,c \in \GF{q}^*$ and $d,e,f,g \in \GF{q}$.
\[ \begin{pmatrix}
1 & 0 & 0\\
d & a & 0\\
e & f & b
\end{pmatrix} \quad \quad
\begin{pmatrix}
1 & 0 \\
g & c
\end{pmatrix}  \]

The induced elements of $L_{\beta,\gamma}$ are those of $L_{\beta,\ell}$ with $d=0$; of $L_{\beta,\gamma}$ fixing all of $\Gamma_0(\ell)$ are those of $L_{\beta,\ell}$ with $a=1$ and $d=0$; and the induced action of $L_{\beta,\ell}$ on $\Gamma_0(\ell)\backslash \{\beta, \gamma\}$ are those of $L_{\beta,\ell}$'s action on $\Gamma_0(\ell)\backslash \{\beta\}$ with $g=0$. \end{proof}

Let $\t{\beta, \gamma}$ denote an involution in $L_{\beta,\gamma}$, whose support is restricted to the points on $\ell \setminus \{ \beta, \gamma\}$. By Lemma \ref{cyclicaction} its action on $\Gamma_0(\ell) \backslash \{\beta,\gamma\}$ is well-defined and $\t{\beta,\gamma}$ will be the product of $(q-1)/2$ transpositions. In the case of $\beta=\langle e_1 \rangle$ and $\gamma=\langle e_2 \rangle$ the action of $\t{\beta,\gamma}$ on $\ell \backslash \{\beta, \gamma\}$ is as follows
\[\t{\beta,\gamma} : \langle ae_1 + be_2 \rangle \mapsto  \langle ae_1 + be_2 \rangle \begin{pmatrix}1 & 0 \\ 0 & -1 \end{pmatrix} = \langle ae_1-be_2 \rangle \;\;\; \text{ for } a,b \in \GF{q}^*.\]

We mirror the 15-Puzzle construction as follows. Place numbered counters on $q^2+q$ of the points and leave the remaining point empty.  Following \cite{Conwaysurvey} we call this empty point the \emph{hole}. 
For each choice of counter $t$ we can define an elementary move. Let $\beta \in \P{q}$ be the hole and $\gamma \in \P{q}$ be the point covered by $t$. Then the corresponding \emph{elementary move}, $\h{\beta,\gamma}$, is given by moving the counter $t$ to $\beta$ (thus leaving $\gamma$ empty), and applying $\t{\beta,\gamma}$. 
It is then well defined to apply $\h{\delta,\epsilon}$ after $\h{\beta,\gamma}$ if and only if $\gamma=\delta$.

As introduced in Section \ref{Sec:Intro}, the groupoid $\Pseudo{q}$ consists of well defined concatenations of elementary moves. That is, the elements $\h{\beta_0,\beta_1,\ldots, \beta_s}:=\h{\beta_0,\beta_1}\h{\beta_1,\beta_2}\cdots \h{\beta_{s-1},\beta_s}$. The sequence $\beta_0,\beta_1,\ldots, \beta_s$ gives the path that the hole takes through $\P{q}$.
\section{Proof of Theorem \ref{main}}
Throughout this section, fix $q>3$ an odd prime power and let $G:=\PseudoG{q}$.
Recall that $\alpha$ is a fixed point of $\P{q}$. 
For $\beta, \gamma \in \Omega$ with $\beta \ne \gamma$, put
$$\x{\beta, \gamma} = \h{\alpha, \beta}\h{\beta, \gamma}\h{\gamma,\alpha}.$$
Observe that $G$ is generated by these $\x{\beta,\gamma}$.
We begin by proving a lemma about the elements $\x{\beta,\gamma}$ where $\alpha, \beta$ and $\gamma$ do not all lie on a common line of $\PP{q}$.

\begin{lemma}\label{notcollinear} Suppose that $\beta, \gamma \in \Omega$ are distinct, and that $\alpha, \beta, \gamma$ are not collinear in $\mathbb{P}(q)$. Then 
\[\x{\beta, \gamma} = (\beta,\gamma)\t{\alpha,\beta}\t{\beta,\gamma}\t{\gamma,\alpha}\]
with $\x{\beta,\gamma}$ a product of $\frac{3q - 1}{2}$ pairwise disjoint transpositions.
\end{lemma}

\begin{proof} By definition 
\[\h{\alpha,\beta} = (\alpha,\beta)\t{\alpha,\beta}, \quad \h{\beta,\gamma} = (\beta,\gamma)\t{\beta,\gamma} \quad \text{and} \quad \h{\gamma,\alpha} = (\gamma,\alpha)\t{\gamma,\alpha}\] where $\t{\alpha,\beta}, \t{\beta,\gamma}$ and $\t{\gamma,\alpha}$ are each products of $\frac{q-1}{2}$ pairwise disjoint transpositions. By assumption $\alpha,\beta$ and $\gamma$ are not collinear, and so the supports of $\t{\alpha,\beta}$, $\t{\beta,\gamma}$ and $\t{\gamma, \alpha}$ are pairwise disjoint. In addition, $(\alpha,\beta), (\beta,\gamma)$ and $(\gamma,\alpha)$ all commute with each of $\t{\alpha,\beta}, \t{\beta,\gamma}$ and $\t{\gamma,\alpha}$. Hence from $(\alpha,\beta)(\beta,\gamma)(\gamma,\alpha)  = (\alpha)(\beta,\gamma)$, we obtain the required expression for $\x{\beta,\gamma}$. The number of transpositions is $3(\frac{q-1}{2} ) + 1 = \frac{3q - 1}{2}$.
\end{proof}

We now consider the case of $\alpha, \beta$ and $\gamma$ being collinear. We begin with some notation.
For a line $\ell\in \L{q}$ define the following subgroup of $G$,
\[G(\ell) = \left\langle \x{\beta,\gamma} \;| \; \beta, \gamma \in \Gamma_0(\ell) \setminus \{ \alpha \}, \; \beta \neq \gamma \right\rangle.\]
Recall that $L = \PGL_3(q)$ and $L_{\alpha,\ell} = \Stab_L(\alpha) \cap \Stab_L(\ell)$.

\begin{lemma}\label{collinear} 
Suppose that $\beta, \gamma \in \Omega$ are distinct, and that $\alpha, \beta$ and $ \gamma$ are collinear in $\mathbb{P}(q)$. Let $\ell$ be the line they all lie on and let $g \in L_{\alpha,\ell}$. Then the following hold.
\begin{enumerate}[label=\rm{(\roman*)}]
\item $g^{-1}\x{\beta, \gamma}g = \x{\beta^g,\gamma^g}$
\item $\x{\beta, \gamma} \neq 1$
\item $L_{\alpha,\ell}$ normalizes $G(\ell)$
\item $G(\ell)$ acts transitively on $\Gamma_0(\ell) \setminus \{\alpha\}$
\end{enumerate}
\end{lemma}

\begin{proof} 

\begin{enumerate}[label=\rm{(\roman*)}]
\item Since $\t{\beta, \gamma}^g \in (\Stab_{L_{\alpha,\ell}}(\beta) \cap \Stab_{L_{\alpha,\ell}}(\gamma))^g= \Stab_{L_{\alpha,\ell}}(\beta^g) \cap \Stab_{L_{\alpha,\ell}}(\gamma^g)$ and $\t{\beta,\gamma}^g$ induces the same action as $\t{\beta^g,\gamma^g}$ on $\Gamma_0(\ell) \setminus \{ \beta^g, \gamma^g \}$, it follows that $\t{\beta,\gamma}^g
= \t{\beta^g,\gamma^g}.$ Hence 
\[\h{\beta,\gamma}^g =((\beta,\gamma)\t{\beta,\gamma})^g = (\beta^g,\gamma^g)\t{\beta^g,\gamma^g} = \h{\beta^g,\gamma^g}.\] 
Since $g$ fixes $\alpha$, it follows as above that $\h{\alpha,\beta}^g = \h{\alpha,\beta^g}$ and $\h{\gamma,\alpha}^g = \h{\gamma^g,\alpha}$. Therefore the result follows.

\item Fix $\alpha=\langle e_1+e_2 \rangle$. By part (i) we may assume that $\beta = \langle e_1 \rangle$ and $\gamma=\langle e_2 \rangle$. For $a,b \in \GF{q}$ we have the following.
\begin{align*}
\langle ae_1+be_2 \rangle^{\h{\alpha,\beta}} &= \begin{cases} 
\langle e_1 \rangle & \text{ if } a=b^{-1},\\
\langle e_1+e_2 \rangle & \text{ if } b=0, \\
\langle (a-2b)e_1-be_2 \rangle & \text{ otherwise.}
\end{cases}\\ \\
\langle ae_1+be_2 \rangle^{\h{\beta,\gamma}} &= \begin{cases} 
\langle e_2 \rangle & \text{ if } b=0, \\
\langle e_1 \rangle & \text{ if } a=0,\\
\langle ae_1-be_2 \rangle \quad\quad\quad\quad & \text{ otherwise.}
\end{cases} \\ \\
\langle ae_1+be_2 \rangle^{\h{\gamma,\alpha}} &= \begin{cases} 
\langle e_1+e_2 \rangle & \text{ if } a=0, \\
\langle e_2 \rangle & \text{ if } a=b^{-1},\\
\langle ae_1+(2a-b)e_2 \rangle & \text{ otherwise.}
\end{cases} 
\end{align*}
First assume that $q$ is not a 3-power. Then
\[\langle 2e_1+e_2 \rangle^{\h{\alpha,\beta}\h{\beta,\gamma}\h{\gamma,\alpha}}=\langle e_2 \rangle^{\h{\beta,\gamma}\h{\gamma,\alpha}}= \langle e_1 \rangle^{\h{\gamma,\alpha}}= \langle e_1+2e_2\rangle.\]
Hence $\x{\beta,\gamma}$ maps $\langle 2e_1+e_2 \rangle$ to $\langle e_1+2e_2 \rangle=\langle 2e_1+4e_2\rangle$. Since $q$ is not a 3-power, it follows that $\langle 2e_1+e_2 \rangle \neq \langle 2e_1+4e_2\rangle$. Therefore $\x{\beta,\gamma} \neq 1$

Next assume that $q=3^r$ for some $r>1$. Let $\omega$ be a primitive root of unity in $\GF{q}$, and so in particular $\omega \neq -1$. Then
\begin{align*} 
\langle e_1+\omega e_2 \rangle^{\h{\alpha,\beta}\h{\beta,\gamma}\h{\gamma,\alpha}}&=\left\langle (1+\omega)e_1-\omega e_2 \right\rangle^{\h{\beta,\gamma}\h{\gamma,\alpha}}\\
&= \langle (1+\omega)e_1+\omega e_2 \rangle^{\h{\gamma,\alpha}}\\
&= \langle (1+\omega)e_1+(-1+\omega) e_2 \rangle.
\end{align*}
Hence $\x{\beta,\gamma}$ maps $\langle e_1+\omega e_2 \rangle=\langle (1+\omega)e_1+(\omega+\omega^2)e_2\rangle$ to $\langle (1+\omega)e_1+(-1+\omega) e_2 \rangle$. Since $q \geq 9$, we have that $w^{2} \neq -1$ and so $\langle (1+\omega)e_1+(\omega+\omega^2)e_2\rangle \neq \langle (1+\omega)e_1+(-1+\omega) e_2 \rangle$. Therefore, it this case we also have $\x{\beta,\gamma} \neq 1.$

\item It follows by part (i) that the generating set for $G(\ell)$ is invariant under conjugation by $L_{\alpha,\ell}$.

\item By Lemma \ref{cyclicaction}\ref{beta,ell}, $L_{\alpha,\ell}$ is 2-transitive on $\Gamma_0(\ell) \setminus \{\alpha\}$, and so $L_{\alpha,\ell}G(\ell)$ is also. Recall that 2-transitive groups are primitive, and non-trivial normal subgroups of primitive groups are transitive. Hence $G(\ell) \trianglelefteq L_{\alpha,\ell}G(\ell)$ is transitive on $\Gamma_0(\ell)\backslash \{\alpha\}$. 
\qedhere
\end{enumerate}
\end{proof}

\begin{remark}
\begin{enumerate}[label=\rm{(\roman*)}]
\item If instead $q=3$, Lemma \ref{collinear}(ii) does not hold. In a certain sense, this allows for the emergence of $\M_{12}$. 
\item The elements $\x{\beta,\gamma}$ in Lemma \ref{collinear} have a surprising range of possible cycle types on $\Gamma_0(\ell) \setminus \{\alpha \}$. Here is a sample of the cycle types for some small values of $q$.

\begin{tabular}{|c||c|c|c|c|c|c|c|c|c|c|c|}
\hline
$q$ & 5 & 7 & 9 & 11 & 13 & 17 & 19 & 23 & 25 & 27 & 29\\ \hline
cycle type & $1^1.4^1$ & $7^1$ & $1^5.4^1$ & $1^2.3^3$ & $2^3.7^1$ & $3^3.8^1$ & $1^2.17^1$ & $23^1$ & $1^1.4^1.5^4$ & $1^3.4^6$ & $1^2.13^1.14^1$ \\ \hline
\end{tabular}
\end{enumerate}
\end{remark}

\begin{lemma}\label{primitive} $G$ acts primitively on $\Omega$.
\end{lemma}
\begin{proof}  For $\beta \in \Omega$, let $\ell_{\beta}=\langle \alpha, \beta \rangle$ and $\Gamma_{\beta}=\Gamma_0(\ell_{\beta}) \cap \Omega = \Gamma_0(\ell_{\beta})\backslash \{\alpha\}$.

We claim that if $G$ is imprimitive, then there exists a non-trivial block $B$ and a point $\beta \in B$ with $\Gamma_{\beta} \subseteq B \not\subseteq \Gamma_{\beta}$, from which we will derive a contradiction. \medskip

First we find a non-trivial block $B$ containing a point $\beta$ such that $B \not\subseteq \Gamma_{\beta}$.
If this holds for all blocks, then we are done. Hence assume that $C$ is a non-trivial block containing a point $\gamma$ such that $C \not\subseteq \Gamma_{\gamma}$. Then there exists $\delta \in C \backslash \{\gamma\} \subseteq \Gamma_{\gamma}$ and $\epsilon \in \Omega \backslash\Gamma_{\gamma}$. 
Now $C^{\x{\gamma,\epsilon}}$ contains $\delta^{\x{\gamma,\epsilon}} \in \Gamma_{\gamma}$ and $\gamma^{\x{\gamma,\epsilon}} = \epsilon \notin \Gamma_{\gamma}$. Hence let $B=C^{\x{\gamma,\epsilon}}$ and $\beta = \delta^{\x{\gamma,\epsilon}}$.\medskip

We now show $\Gamma_{\beta} \subseteq B$. Let $\zeta \in \Gamma_B$, by the above we may let $\eta \in B \backslash \Gamma_B$. By Lemma \ref{collinear}(iv) there exists $g \in G(\ell_{\beta})$ with $\beta^g=\zeta$. Hence $\zeta = \eta^{\x{\beta,\eta}g} \in B^{\x{\beta,\eta}g}$. Since $G(\ell_{\beta})$ fixes $\Omega\backslash \Gamma_{\beta}$ pointwise, it follows that $\beta^{\x{\beta,\eta}g}=\eta \in B \cap B^{\x{\beta,\eta}g}$. Hence $\zeta\in B$.\medskip

Hence the claim holds. Let $\theta \in \Omega \backslash B$. Since $\Gamma_{\beta} \subseteq B$, it follows that $\ell_{\theta} \neq \ell_{\beta}$, and so $\x{\beta,\theta}$ interchanges $\beta$ and $\theta$ and fixes $\Gamma_{\beta} \backslash \{\beta\}$ setwise. Therefore $\theta  \in B^{\x{\beta,\theta}}\backslash B$ and $\Gamma_{\beta} \backslash \{\beta\} \subseteq B \cap B^{\x{\beta,\theta}}$, a contradiction.
\end{proof}

We can now complete the proof of Theorem \ref{main}.

\begin{proof}[Proof of Theorem \ref{main}]
By Lemma \ref{primitive} $G$ acts primitively on $\Omega$. Let $\beta,\gamma \in \Omega$ be distinct and such that $\alpha,\beta, \gamma$ are collinear. Then $\x{\beta,\gamma}$ is an element of $G$ with support size at most $q+1$. It is a simple exercise to show that $q+1<2(\sqrt{|\Omega|}-1)$, and so $\Alt(n) \leq G$ by \cite[Corollary 3]{Jordan}. Since $\h{\beta,\gamma}$ is a product of $\frac{q + 1}{2}$ transpositions, $\x{\beta,\gamma}$ is an even permutation if and only if $\frac{q + 1}{2}$ is even, which is equivalent to $q \equiv 3 \bmod 4$ and completes the proof of the theorem.
\end{proof}

\end{document}